\documentclass[11pt]{amsart}

\usepackage{amsmath}
\usepackage{amsfonts}
\usepackage{amssymb}
\usepackage{amsthm}
\usepackage{hyperref}
\usepackage{enumerate}
\usepackage{cleveref}
\usepackage{mathrsfs}

\newtheorem{theorem}{Theorem}[section]
\newtheorem{lemma}[theorem]{Lemma}

\newtheorem{corollary}[theorem]{Corollary}

\newtheorem*{remark*}{Remark}

\Crefname{conjecture}{Conjecture}{Conjectures}

\theoremstyle{plain}

\theoremstyle{plain}

\newcommand{\N}{\mathbb{N}}

\newcommand{\Z}{\mathbb{Z}}

\newcommand{{\D}}{\delta}

\newcommand{\bbH}{\mathbb{H}}

\numberwithin{equation}{section}
\numberwithin{table}{section}

\author{Michael H. Mertens, Larry Rolen}
\address{Department of Mathematics and Computer Science, Emory University, 400 Dowman Drive, 30322 Atlanta, GA } 
\email{michael.mertens@emory.edu}
\address{Mathematics Institute
University of Cologne,
Gyrhofstr. 8b
50931 Cologne}
\email{lrolen@math.uni-koeln.de}

\title{Lacunary recurrences for Eisenstein series}
\thanks{
The second author thanks the University of Cologne and the DFG for their generous support via the University of Cologne postdoc grant DFG Grant D-72133-G-403-151001011, funded under the Institutional Strategy of the University of Cologne within the German Excellence Initiative.}

\begin{document}
\begin{abstract}
Using results from the theory of modular forms, we reprove and extend a result of Romik about lacunary recurrence relations for Eisenstein series.
\end{abstract}
\maketitle

\section{Introduction}
It is a classical result from the theory of elliptic functions that the Eisenstein series 
\begin{equation}
\label{GkDef}
G_k(\tau):=\sum\limits_{(m,n)\in\Z^2\setminus\{(0,0)\}}(m\tau+n)^{-k}=2\zeta(k)\left(1-\frac{2k}{B_k}\sum\limits_{n=1}^\infty \sum\limits_{d|n}d^{k-1}e^{2\pi i\tau}\right),
\end{equation}
where $k\geq 2$ is an even integer\footnote{Note that for $k=2$, we have to fix a certain order of summation to ensure convergence of the defining double series and that the Fourier expansion given is still valid.}, $\tau$ is an element of the complex upper half-plane $\bbH$, and $B_k$ denotes the $k$th Bernoulli number, satisfy the following recurrence relation
\begin{equation}\label{Hurwitz}
(n-3)(2n-1)(2n+1)G_{2n}=3\sum\limits_{\substack{p,q\geq 2 \\ p+q=n}}(2p-1)(2q-1)G_{2p}G_{2q}.
\end{equation}

From the analytic properties of the zeta function $\omega(s)$, which is basically a special case of a Witten zeta function and is essentially a Dirichlet series generating function of the dimensions of irreducible representations of $\operatorname{SU}(3)$, Romik \cite{Romik} derived a new recurrence relation, given by
\begin{equation}
\label{Romik}
G_{6n+2}=\frac{1}{6n+1}\cdot\frac{(4n+1)!}{(2n)!^2}\sum\limits_{k=1}^n \frac{\binom{2n}{2k-1}}{\binom{6n}{2n+2k-1}}G_{2n+2k}G_{4n-2k+2}.
\end{equation}
The most striking difference between the recurrences \eqref{Hurwitz} and \eqref{Romik} is that in \eqref{Romik}, only about a third of the previous Eisenstein series are needed, while in \eqref{Hurwitz}, all Eisenstein series occur.

In the end of \cite{Romik}, Romik asked for a direct proof of \eqref{Romik} using the theory of modular forms. Here, we provide such a proof, and in particular show that Romik's example is a natural, and especially symmetric, instance of general relations among products of two Eisenstein series which have been classified in \cite{HST} and which are closely related to the theory of period polynomials (see \cite{Popa,Zagier}). In particular, we show the following. 
\begin{theorem}\label{theo6n2}
For all $n\in\N$, \eqref{Romik} holds.
\end{theorem}
As further examples of such identities, we record two additional lacunary recurrences, i.e. recurrences which use relatively few terms, for the Eisenstein series $G_{6n}$ and $G_{6n+4}$, which we give in the following two theorems. 
\begin{theorem}\label{theo6n}
For all $n\geq2$, we have
\begin{align*}
&\binom{6n+1}{2n}G_{6n}\\
=&\sum\limits_{k=1}^{n}\left[\binom{2n+2k-1}{2n}\binom{4n-2k-1}{2n}+2 \binom{2n+2k-1}{2n}\binom{4n-2k-1}{2n-2}\right]G_{2n+2k}G_{4n-2k}.
\end{align*}
\end{theorem}
\begin{theorem}\label{theo6n4}
For all $n\in\N$, we have
\begin{align*}
&\left\{\binom{6n+3}{2n+2}+2\binom{6n+3}{2n}\right\}G_{6n+4}\\
=&\sum\limits_{k=1}^{n+1}\left[\binom{2n+2k-1}{2n}\binom{4n-2k+3}{2n}+2\binom{2n+2k-1}{2n}\binom{4n-2k+3}{2n+2}\right]G_{2n+2k}G_{4n-2k+4}.
\end{align*}
\end{theorem}
As an immediate consequence, by considering only the constant terms in the Eisenstein series above, one recovers several of the lacunary recurrences for Bernoulli numbers which were systematically studied by Agoh and Dilcher \cite{AD1,AD2}.
As examples of the new recurrences proven, we offer the following special cases of Theorems \ref{theo6n} and \ref{theo6n4}.

\begin{align*}
11G_{10}&=5G_4G_6,\\
143G_{12}&=42G_4G_8+25G_6^2,\\
221G_{16}&=60G_6G_{10}+49G_8^2,\\
323G_{18}&=55G_6G_{12}+105G_8G_{10},\\
7429G_{22}&=1001G_8G_{14}+2706G_{10}G_{12},\\
2185G_{24}&=182G_8G_{16} + 546G_{10}G_{14} + 363G_{12}^2.
\end{align*}

\subsection*{Acknowledgements}
The authors would like to thank Ken Ono for suggesting this project as well as Kathrin Bringmann, Steven J. Miller, Dan Romik, and the anonymous referee for helpful comments.

\section{Linear relations among Eisenstein series}
All linear relations among products of two Eisenstein series $G_iG_j$ and the weight $i+j$ Eisenstein series have been classified in Theorem 1 of \cite{HST}, which we recall here. Their proof relies on partial fraction decompositions and extends previous work of Zagier \cite{Zagier} and Popa \cite{Popa}. To explain their results, define for integers $r,s\geq 2$ the function
\[P_{r,s}:=G_rG_s+\frac{\delta_{2,r}}{s}G_s'+\frac{\delta_{2,s}}{r}G_r',\]
where the $'$ denotes the renormalized derivative $\tfrac{1}{2\pi i}\tfrac{d}{d\tau}$, $\delta_{i,j}$ is the usual Kronecker delta symbol, and we set $G_r:=0$ if $r$ is odd. With this, we can state Theorem 1 in \cite{HST} as follows.
\begin{theorem}
Let $r,s,t\geq 1$ be integers such that $k:=r+s+t-1\geq 4$. Then we have
\begin{equation}\label{Relations}
\begin{aligned}
0=&\sum\limits_{i+j=k}\binom{i-1}{t-1}\binom{j-1}{s-1}(-1)^{i+r}(P_{i,j}-(-1)^jG_k)\\
+&\sum\limits_{j+h=k}\binom{j-1}{r-1}\binom{h-1}{t-1}(-1)^{j+s}(P_{h,j}-(-1)^hG_k)\\
+&\sum\limits_{h+i=k}\binom{h-1}{s-1}\binom{i-1}{r-1}(-1)^{h+t}(P_{h,i}-(-1)^iG_k).
\end{aligned}
\end{equation}
All linear relations among $G_k$ and $P_{2j,k-2j}$, $j=1,...,\lfloor\tfrac k4\rfloor$ are of the form \eqref{Relations}.
\end{theorem}
In particular, the relations in \eqref{Romik} must arise as specializations of the previous theorem. Indeed, by setting $r=s=t=2n+1$ for $n\in\N$, hence $k=6n+2$, we obtain the following immediate consequence.
\begin{corollary}\label{cor}
The following identity holds for all $n\in\N$,
\begin{equation*}
\begin{split}
& \sum\limits_{k=1}^{2n+1}\binom{2n+k-1}{2n}\binom{4n-k+1}{2n} G_{6n+2} =
\\
&
\sum\limits_{k=1}^n \binom{2n+2k-1}{2n}\binom{4n-2k+1}{2n}G_{2n+2k}G_{4n-2k+2}
.
\end{split}
\end{equation*}
\end{corollary}
\begin{proof}
By our choice of parameters, all the sums on \eqref{Relations} are equal. Furthermore, the product $\binom{i-1}{2n}\binom{6n-i+1}{2n}$ occuring in the sum is zero unless we have $2n+1\leq i\leq 4n+1$. Now we make the index shift $i\mapsto k+2n$ and bring the terms involving only $G_k$ to the left-hand side to obtain
\begin{equation*}
\begin{split}
& \sum\limits_{k=1}^{2n+1}\binom{2n+k-1}{2n}\binom{4n-k+1}{2n} G_{6n+2} =
\\
&
\sum\limits_{k=1}^{2n+1}\binom{2n+k-1}{2n}\binom{4n-k+1}{2n} (-1)^{k}P_{2n+k,4n-k+2}
.
\end{split}
\end{equation*}
Since by definition $P_{r,s}$ is identically zero if one of $r$ and $s$ is odd, we obtain the desired result.
\end{proof}
Theorem 1.1 is a consequence of the following identity for binomial coefficients.
\begin{lemma}
Let $n\geq 1$ and define 
\[B(n):=\sum\limits_{j=1}^{2n+1}\binom{2n+j-1}{2n}\binom{4n-j+1}{2n}.\]
Then we have that
\[\frac 1{B(n)}\binom{2n+2k-1}{2n}\binom{4n-2k+1}{2n}=\frac{1}{6n+1}\cdot\frac{(4n+1)!}{(2n)!^2}\frac{\binom{2n}{2k-1}}{\binom{6n}{2n+2k-1}}\] 
for all $1\leq k\leq n$.
\end{lemma}
\begin{proof}
After simplifying the claim by canceling out terms in the binomial expressions, we find that the claim is equivalent to the identity $
B(n)=\binom{6n+1}{2n}.$
We use the following identity due to Hagen and Rothe (see \cite{Chu}), which is valid for $a,b,c,k\in\N$ whenever the denominator doesn't vanish:
\begin{equation}\label{binom}
\sum_{j=0}^k\frac{a}{(a+bj)}\binom{a+bj}{j}\binom{c-bj}{k-j}=\binom{a+c}{k}.
\end{equation}
Setting $k=2n$, $a=2n+1, b=1, c=4n$, we find
\[
\binom{6n+1}{2n}
=
\sum_{j=0}^{2n}\frac{(2n+1)}{(2n+j+1)}\binom{2n+1+j}{j}\binom{4n-j}{2n-j}
=
\sum_{j=0}^{2n}
\binom{2n+j}{2n}\binom{4n-j}{2n},
\]
which is equivalent to the claim after a shift in $j$. 
\end{proof}

\noindent\emph{Remark:} The identity in Lemma 2.2 of \cite{Romik}, shown basically by the celebrated Wilf-Zeilberger method, is equivalent to the one shown here with more elementary methods.
\medskip

In the above proof, we note that the Eisenstein series of weight $6n+2$ are particularly special as these correspond to a choice of parameters in Theorem 2.1 where $r=s=t$. To illustrate the power of Theorem 2.1, here we offer similar recurrences for the other residue classes modulo $6$. In order to obtain recurrences which are ``as lacunary as possible'', one easily sees from \eqref{Relations}, that for a given weight $k$, one has to choose $r,s,t$ all odd and as large as possible, since then most of the binomial coefficients in the sums will vanish. 
\begin{proof}[Proof of Theorem \ref{theo6n}]
Choosing $(r,s,t)=(2n-1,2n+1,2n+1)$, and hence $k=6n$ (for $n\geq 2$), in \eqref{Relations}, we obtain
\begin{equation*}
\begin{aligned}
0=&\sum\limits_{i=1}^{6n}\binom{i-1}{2n}\binom{6n-i-1}{2n}(-1)^{i}(P_{i,6n-i}-(-1)^iG_{6n})\\
+&\sum\limits_{j=1}^{6n}\binom{j-1}{2n-2}\binom{6n-j-1}{2n}(-1)^{j}(P_{j,6n-j}-(-1)^jG_{6n})\\
+&\sum\limits_{h=1}^{6n}\binom{h-1}{2n}\binom{6n-h-1}{2n-2}(-1)^{h}(P_{h,6n-h}-(-1)^hG_{6n}).
\end{aligned}
\end{equation*}
By reversing the order of summation in the third of these sums, we see that it equals the second one. If we now omit the terms where the binomial coefficients vanish (i.e. those where $i\leq 2n$ or $i\geq 4n$ resp. $j\leq 2n-2$ or $j\geq 4n$) and shift the summation as in the proof of Corollary \ref{cor}, we obtain the identity
\begin{equation*}\label{id6n1}
\begin{aligned}
&\left\{\sum\limits_{j=1}^{2n-1}\binom{2n+j-1}{2n}\binom{4n-j-1}{2n} + 2\sum\limits_{j=1}^{2n+1} \binom{2n+j-3}{2n-2}\binom{4n-j+1}{2n}\right\}G_{6n}\\
=&\sum\limits_{k=1}^{n-1}\binom{2n+2k-1}{2n}\binom{4n-2k-1}{2n}G_{2n+2k}G_{4n-2k}\\
&\hspace{2cm} +2\sum\limits_{k=1}^{n} \binom{2n+2k-1}{2n}\binom{4n-2k-1}{2n-2}G_{2n+2k}G_{4n-2k}.
\end{aligned}
\end{equation*}
 
The sums on the left hand side can be simplified using \eqref{binom} and by straightforward manipulation of binomial coefficients:

For the first sum we can choose $a=2n+1,\:b=1,\:c=4n-2,\:k=2n-2$, so that, after an index shift, we have by \eqref{binom}
\[\sum\limits_{j=1}^{2n-1}\binom{2n+j-1}{2n}\binom{4n-j-1}{2n}=\binom{6n-1}{2n-2}.\]
For the second sum we choose $a=2n-1,\:b=1,\:c=4n,\:k=2n$, which gives the identity
\[\sum\limits_{j=1}^{2n+1} \binom{2n+j-3}{2n-2}\binom{4n-j+1}{2n}=\binom{6n-1}{2n}.\]
Now we compute directly
\begin{align*}
\binom{6n-1}{2n-2}+2\binom{6n-1}{2n}&=\frac{(6n-1)!(2n(2n-1)+2\cdot 4n(4n+1))}{(2n)!(4n+1)!}\\
&=\frac{(6n-1)!(36n^2+6n)}{(2n)!(4n+1)!}=\binom{6n+1}{2n}.
\end{align*}
\end{proof}

In order to prove \Cref{theo6n4}, we choose $(r,s,t)=(2n+1,2n+1,2n+3)$ ($n\geq 1$) in \eqref{Relations} and proceed in the same manner as in the proof of \Cref{theo6n}. Since it would be almost literally the same proof, we omit it here. We just note that in this situation, the factor in front of $G_{6n+4}$ cannot be simplified to a single binomial coefficient.

\section*{Competing interests}
The authors declare that they have no competing interests in the present manuscript.

\end{document}